\documentclass[final,3p,times,12pt]{elsarticle}
\usepackage{}
\usepackage{stmaryrd}
\usepackage{bbm}
\usepackage{txfonts}
\usepackage{amsfonts,dsfont,bm}
\usepackage{color}
\usepackage{xcolor}

\usepackage{graphics}
\usepackage{graphicx}
\usepackage{subfigure}
\usepackage{float}

\usepackage{amssymb}
\usepackage{amsthm}

\usepackage{mathrsfs}
\usepackage{amsmath}
\usepackage{hyperref}

\newtheorem{thm}{Theorem}[section]
\newtheorem{lem}{Lemma}[section]
\newtheorem{cor}{Corollary}[section]

\newtheorem{rem}{Remark}[section]

\begin{document}

\begin{frontmatter}

\title{Numerical analysis of  linear and nonlinear time-fractional subdiffusion equations}

\author[ad1]{Yubo Yang}
\ead{boydman$\_$xm@mail.zjxu.edu.cn}

\author[ad2]{Fanhai Zeng\corref{cor1}}
\cortext[cor1]{Corresponding author}
\ead{fanhaiz@foxmail.com}

\address[ad1]{Nanhu College, Jiaxing University, Jiaxing, Zhejiang 314001, China}
\address[ad2]{Department of Mathematics, National University of Singapore, Singapore 119076, Singapore}

\begin{abstract}
In this paper, a new type of the discrete fractional Gr{\"o}nwall inequality is developed, which
is applied to analyze the stability and convergence of a Galerkin spectral method
for a linear time-fractional subdiffusion equation.
Based on the temporal-spatial error splitting argument technique,
the discrete fractional Gr{\"o}nwall inequality is also
applied to prove the unconditional convergence of a semi-implicit Galerkin spectral method for a nonlinear time-fractional subdiffusion equation.
%
\end{abstract}


\end{frontmatter}

\section{Introduction}\setcounter{equation}{0}
Consider the following   nonlinear time-fractional subdiffusion equation
\begin{equation}
\label{E:1.1}
\left\{\begin{array}{lll}
{_{0}^{C}{\mathcal {D}}}^{\beta}_t u = \mu~\partial^{2}_x u + f(x,t,u), \quad (x,t)\in I\times (0,T), \quad I=(-1,1), \quad T>0,\\ [0.3cm]
u=0,\qquad (x,t)\in \partial I\times (0,T),\\ [0.3cm]
u(x,0)=u_0(x), \qquad x\in I,
\end{array}\right.
\end{equation}
where ${_{0}^{C}{\mathcal {D}}}^{\beta}_t u$ denotes the   Caputo time-fractional  derivative of order $0<\beta<1$  defined by (cf. \cite{Kilbas2006})
\begin{equation}
\label{E:1.2}
{_{0}^{C}{\mathcal {D}}}^{\beta}_t u(x,t) = \frac{1}{\Gamma(1-\beta)} \int_{0}^{t} (t-s)^{-\beta}\partial_s u(x,s) {\rm d}s,
\end{equation}
in which $\Gamma(z)=\int_{0}^{\infty} s^{z-1}e^{-s} {\rm d}s$ is the Gamma function.


Various time-stepping schemes have been developed for discretizing \eqref{E:1.1}. The  time discretization
technique for the time-fractional operator in \eqref{E:1.1} mainly falls into two categories: interpolation and the fractional linear multistep method (FLMM, which is also called the convolution quadrature (CQ)) based on generating functions that can be derived from the linear multistep method for the initial value problem. For example, piecewise linear interpolation yields the widely applied $L1$ method
\cite{lin2007finite,liao2018sharp}. High-order interpolation can also be applied, see \cite{alikhanov2015new,li2015numerical,ZhaoSK15}.
The FLMM \cite{Lubich1986Discretized,Lubich1988Convolution,Lubich2004Convolution} provides another general framework for constructing high-order methods to discretize the fractional integral and derivative operators. The FLMM inherits the stability properties of linear multistep methods for initial value problem, which greatly facilitates the analysis of the resulting numerical scheme, in a way often strikingly opposed to standard quadrature formulas \cite{Lubich2004Convolution}. Up to now, the FLMM has been widely applied to discretize the model \eqref{E:1.1} and its variants.

It is well known that the classical discrete Gr{\"o}nwall inequality plays an important role in the analysis of the numerical methods for time-dependent partial differential equations (PDEs). Due to the lack of a generalized discrete  Gr{\"o}nwall type inequality for the time-stepping methods of the time-dependent fractional differential equations (FDEs), the analysis of the numerical methods for time-dependent FDEs is more complicated. Recently, a discrete fractional Gr{\"o}nwall inequality has been established by Liao et al. \cite{li2016analysis,liao2018sharp,liao2018discrete,liao2018second} for interpolation methods to solve linear and nonlinear time-dependent FDEs. Jin et al. \cite{jin2018numerical} proposed a criterion for showing the fractional discrete Gr{\"o}nwall inequality and verified it for the $L1$ scheme and convolution quadrature generated by backward difference formulas.

Till now, there have been some works on the numerical analysis of nonlinear time-dependent FDEs. The stability and convergence of $L1$ finite difference methods were obtained for a time-fractional nonlinear predator-prey model under the restriction $LT^{\beta} < 1/\Gamma(1-\beta)$ in \cite{yu2015positivity}, where $L$ is the Lipschitz constant of the nonlinear function, depending upon an upper bound of numerical solutions \cite{li2018unconditionally}. Such a condition implied that the numerical results just held locally in time and certain time step restriction condition (see, e.g., \cite{akrivis1999implicit,ford2013numerical}) were also required. Similar restrictions also appear in the numerical analysis for the other fractional nonlinear equations (see, e.g. \cite{li2016efficient,li2016linear}). In order to avoid such a restriction, the temporal-spatial error splitting argument (see, e.g., \cite{li2012mathematical}) is extended to the numerical analysis of the nonlinear time-dependent FDEs (see, e.g., \cite{li2017unconditionally, li2018unconditionally}). Li et al. proposed unconditionally convergent $L$1-Galerkin finite element methods (FEMs) for nonlinear time-fractional Sch{\"o}dinger equations \cite{li2017unconditionally} and nonlinear time-fractional subdiffusion equations\cite{li2018unconditionally}, respectively.

In this paper, we follow the idea in \cite{liao2018sharp} and develop a discrete fractional Gr{\"o}nwall inequality for analyzing the FLMM that arises from the generalized Newton--Gregory formula (GNGF) of order up to two, see \cite{zeng2013use}.
Compared with the approach based on interpolation in \cite{liao2018sharp}, the discrete kernel $P_{k-j}$ (see Lemmas \ref{lem2.2} and \ref{lem2.3} below) originates  from the generating
function that can be obtained exactly, which is much simpler than that in \cite{liao2018sharp}. Based on the discrete fractional Gr{\"o}nwall inequality, the numerical analysis of semi-implicit Galerkin spectral  method for the time-fractional nonlinear subdiffusion problem \eqref{E:1.1} is advanced. The temporal-spatial error splitting argument is used to prove the unconditional stability and convergence of the semi-implicit method.

The main task of this work is to establish the discrete fractional Gr{\"o}nwall type inequality for the stability and convergence analysis of the numerical methods for time-fractional PDEs, the regularity and singularity of the solution at $t=0$ is not considered in detail here; readers can refer to \cite{stynes2017error,zeng2017second}
for the graded mesh method and correction method for resolving the singularity of the solution of the time-fractional PDEs.

The paper is organized as follows. In Section \ref{sec-2},  the discrete fractional Gr{\"o}nwall inequality for the CQ  is developed, which is applied to the numerical analysis for the linear time-fractional PDE. In Section 3, the unconditional convergence of the semi-implicit Galerkin spectral method is proved by combining the discrete fractional Gr{\"o}nwall inequality and the temporal-spatial error splitting argument.
Some conclusion remarks are given in Section 4.

\section{Numerical analysis for the linear equation}\label{sec-2}
\setcounter{equation}{0}
In this section, two numerical schemes are proposed for the linear equation \eqref{E:1.1},
i.e., $f(x,t,u)=f(x,t)$, in which the time direction is approximated by the fractional linear multistep methods (FLMMs) and the space direction is approximated by the Galerkin spectral method.

Let $\{t_k=k\tau\}_{k=0}^{n_T}$ be a uniform partition of the interval $[0,T]$  with a time step size $\tau=T/n_T$, $n_T$ is a positive integer. For simplicity, the solution of \eqref{E:1.1} is denoted by $u(t)=u(x,t)$  if no confusion is caused. For function $u(x,t) \in C(0,T;L^2(I))$, denote $u^k=u^k(\cdot)=u(\cdot,t_k)$ and $u^{k-\theta}=(1-\theta) u^k+ \theta u^{k-1}, \theta \in [0,1].$

Let $\mathbb{P}_N(I)$ be the set of all algebraic polynomials of degree at most $N$ on $I$. Define the approximation space as follows:
\begin{eqnarray*}
V_N^0=\{v:v \in \mathbb{P}_N(I),v(-1)=v(1)=0\}.
\end{eqnarray*}

%


As in \cite{zeng2013use},  we present the time discretization for  \eqref{E:1.1}  as follows:
\begin{eqnarray}
\label{E:2.5}
{D}^{(\beta)}_{\tau} u^k=\mu L_p^{(\beta)} \partial^2_x u^k+L_p^{(\beta)} f^k+R^k,
\end{eqnarray}
where  $f^k=f(x,t_k)$, $R^k$ is   the discretization error in time that will be specified later,
$L_p^{(\beta)}$ and ${D}^{(\beta)}_{\tau}$ are  defined by
\begin{eqnarray}
\label{E:2.3}
L_p^{(\beta)}u^k=\left\{\begin{array}{ll}
u^k, ~~\qquad p=1,\\ [0.3cm]
u^{k-\beta/2}, \quad p=2,
\end{array}\right.
\end{eqnarray}
and
\begin{eqnarray}
\label{E:2.2}
{D}^{(\beta)}_{\tau} u^k=\frac{1}{{\tau}^{\beta}}\sum_{j=0}^k {\varpi}_{k-j}\left(u^j-u^0\right),
\end{eqnarray}
respectively, in which ${\varpi}_{k}$ satisfy the generating function ${\varpi}(z)=(1-z)^{\beta}
=\sum_{k=0}^{\infty}{\varpi}_{k}z^k$.

The   fully discrete schemes for \eqref{E:1.1}   are given as follows: find $u_N^k\in V_N^0$ such that
\begin{equation}\label{E:2.6}
\left\{\begin{array}{ll}
\left({D}^{(\beta)}_{\tau} u_N^k,v\right)+ \mu \left(L_p^{(\beta)}\partial_x u_N^{k},\partial_x v\right)= \left(F_p^k,v\right),\quad k=1,2,\cdots,n_T,  \forall v \in V_N^0,\\ [0.1cm]
u^0=I_N u_0,
\end{array}\right.
\end{equation}
where $F_p^k=I_N\left(L_p^{(\beta)} f^k\right)$  and $I_N$ is the   Legendre--Gauss--Lobatto (LGL) interpolation operator.

\subsection{Discrete fractional Gr{\"o}nwall inequality}

In this subsection, we introduce some useful lemmas and present a  discrete fractional Gr{\"o}nwall inequality that is used in the stability and convergence analysis for \eqref{E:2.6}.
\begin{lem}[{see e.g., \cite{zeng2013use}}]\label{lem2.1}  For $0<\beta<1$, let ${\varpi}_{j}$ be given by ${\varpi}(z)=(1-z)^{\beta}=\sum_{k=0}^{\infty}{\varpi}_{k}z^k$. Then one has
\begin{equation}
\label{E:2.1.1}
\left\{\begin{array}{llllll}
{\varpi}_{j}=(-1)^j {\beta \choose j}=\frac{\Gamma(j-\beta)}{\Gamma(-\beta)\Gamma(j+1)},\\ [0.5cm]
{\varpi}_{0}=1, {\varpi}_{j}<{\varpi}_{j+1}<0,\qquad j \geq 1,\\ [0.5cm]
\sum_{j=0}^{\infty} {\varpi}_{j}=0,\\ [0.5cm]
{\varpi}_{0}=-\sum_{j=1}^{\infty} {\varpi}_{j}>-\sum_{j=1}^{k} {\varpi}_{j}>0,\\ [0.5cm]
b_{k-1}=\sum_{j=0}^{k-1} {\varpi}_{j}>0, \qquad k \geq 1,\\ [0.5cm]
b_{k-1}=\frac{\Gamma(k-\beta)}{\Gamma(1-\beta)\Gamma(k)}=\frac{k^{-\beta}}{\Gamma(1-\beta)}+O(k^{-1-\beta}),\quad k=1,2,\cdots.
\end{array}\right.
\end{equation}
Furthermore, $b_{k}-b_{k-1}={\varpi}_{k}<0$ for $k>0$, i.e., $b_{k}<b_{k-1}$.
\end{lem}
\begin{lem}\label{lem2.2} For $0<\beta<1$, let ${\varpi}_{j}$ be given by ${\varpi}(z)=(1-z)^{\beta}=\sum_{k=0}^{\infty}{\varpi}_{k}z^k$, ${\varrho}_{j}$ be given by $\varrho(z)=(1-z)^{-\beta}=\sum_{k=0}^{\infty}{\varrho}_{k}z^k$, and
$${\vartheta}_m:=\sum_{j=0}^{m}{\varpi}_{j}{\varrho}_{m-j},\quad m=0,1,2,\cdots.$$
Then one has
\begin{equation}
\label{E:2.1.2}
\left\{\begin{array}{llll}
{\varrho}_{j}=(-1)^j{-\beta \choose j}=\frac{(-1)^j\Gamma(j+\beta)}{\Gamma(\beta)\Gamma(j+1)},\qquad j \geq 0,\\ [0.5cm]
{\varrho}_{0}=1,\quad{\varrho}_{j}>{\varrho}_{j+1}>0,\qquad j \geq 1,\\ [0.5cm]
{\varrho}_{j} \leq (j+1)^{\beta-1},\qquad j \geq 0,\\ [0.5cm]
{\varrho}_{j} \leq j^{\beta-1},\qquad j \geq 1,\\ [0.5cm]
\sum_{j=0}^{k-1} {\varrho}_{j}=\frac{\Gamma(k+\beta)}{\Gamma(1+\beta)\Gamma(k)} \leq \frac{k^{\beta}}{\beta},\quad k=1,2,\cdots,\\ [0.5cm]
{\vartheta}_0=1,\quad{\vartheta}_m=0, \qquad m \geq 1.
\end{array}\right.
\end{equation}
\end{lem}
\begin{proof} By the binomial theorem, we can easily get the first two lines in \eqref{E:2.1.2}.
The middle three lines in \eqref{E:2.1.2} can be derived by the technique in \cite[p.6]{jin2018numerical}. The last line in \eqref{E:2.1.2} can be deduced from the following relation
$$1={\varpi}(z)\varrho(z)=\left(\sum_{j=0}^{\infty}{\varrho}_jz^j\right)
\left(\sum_{j=0}^{\infty}{\varpi}_{j}z^j\right)=\sum_{j=0}^{\infty}{\vartheta}_{j}z^j.$$
The proof is complete.
\end{proof}
\begin{lem}\label{lem2.3} Let $P_{k-j}:={\tau}^{\beta}{\varrho}_{k-j}$. For $0<\beta<1$ and any real $\mu>0$, one has
\begin{eqnarray}
\label{E:2.1.3}
\mu \sum_{j=0}^{k-1} P_{k-j} E_{\beta}(\mu t_j^{\beta}) \leq E_{\beta}(\mu t_k^{\beta})-1, \qquad 1\leq k \leq n_T,
\end{eqnarray}
where $E_{\beta}$ denotes the Mittag--Leffler function that is defined by
\begin{eqnarray}
\label{E:2.1.4}
E_{\beta}(z)=\sum_{l=0}^{\infty} \frac{z^l}{\Gamma(1+l\beta)}.
\end{eqnarray}
\end{lem}
\begin{proof}
Denote $v_l(t)=\frac{t^{l\beta}}{\Gamma(1+l\beta)}$. It is easy to obtain
\begin{eqnarray*}
\sum_{j=0}^{k-1} P_{k-j} \left({_{0}^{C}{\mathcal {D}}}^{\beta}_t v_l\right)(t_j)=\sum_{j=0}^{k-1} P_{k-j} v_{l-1}(t_j).
\end{eqnarray*}
By the fourth inequality in Lemma \ref{lem2.2} and Lemma 3.2 in \cite{li2013finite}, we obtain
\begin{eqnarray*}
\sum_{j=0}^{k-1} P_{k-j} v_{l-1}(t_j)
&=&\sum_{j=0}^{k-1} {\varrho}_{k-j} \left[\frac{j^{(l-1)\beta}{\tau}^{l\beta}}{\Gamma((l-1)\beta+1)}\right]\\
&=&\sum_{j=0}^{k-1}\left[\frac{{\varrho}_{k-j}j^{(l-1)\beta}}
{\Gamma((l-1)\beta+1)}\cdot\frac{\Gamma(l\beta+1)}{k^{l\beta}}\right] v_l(t_k) \\
&\leq&v_l(t_k)\sum_{j=0}^{k-1}\left[\frac{{(k-j)}^{\beta-1}j^{(l-1)\beta}}
{\Gamma((l-1)\beta+1)}\cdot\frac{\Gamma(l\beta+1)}{k^{l\beta}}\right] \\
&\leq&v_l(t_k)
\end{eqnarray*}

Therefore, one has
\begin{eqnarray*}
\sum_{l=1}^{m} {\mu}^l \sum_{j=0}^{k-1} P_{k-j} v_{l-1}(t_j) \leq \sum_{l=1}^{m} {\mu}^l v_l(t_k).
\end{eqnarray*}
Interchanging the sums on the left-hand side of the above inequality
and letting $m\rightarrow \infty$
yields the desired result. The proof is completed.
\end{proof}

We now present the discrete fractional Gr{\"o}nwall inequality in the next theorem.
\begin{thm}\label{thm:2.1} (\rm discrete fractional Gr{\"o}nwall inequality). Let $P_{k-j}:={\tau}^{\beta}{\varrho}_{k-j}$,   $0\leq \theta \leq 1$, and   $\{g^k\}_{k=0}^{n_T}$ and $\{\lambda_l\}_{l=0}^{n_T-1}$ be given non-negative sequences. Assume that there exists a constant $\lambda$ (independent of the time step size) such that $\lambda \geq \sum_{l=0}^{k-1} \lambda_l$, and that the maximum time step size $\tau$ satisfies
\begin{eqnarray}
\label{E:2.1.5}
\tau \leq \frac{1}{\sqrt[\beta]{2\lambda(1+\beta)}}.
\end{eqnarray}
Then, for any non-negative sequence $\{v^k\}_{k=0}^N$ satisfying
\begin{eqnarray}
\label{E:2.1.6}
{D}^{(\beta)}_{\tau} (v^k)^2 \leq \sum_{l=1}^k \lambda_{k-l} (v^l)^2+v^{k-\theta}g^{k-\theta},
\qquad   1 \leq k \leq n_T,
\end{eqnarray}
it holds that
\begin{eqnarray}
\label{E:2.1.7}
v^k \leq 2E_{\beta}(2\lambda t_k^{\beta})\left(v^0+\max_{1\leq m \leq k} \sum_{j=0}^m P_{m-j}g^{j-\theta}\right), \qquad 1 \leq k \leq n_T.
\end{eqnarray}
\end{thm}
\begin{proof}
By   \eqref{E:2.2} and the last line in Lemma \ref{lem2.2}, one has
\begin{eqnarray}\label{E:2.1.8}
\sum_{m=0}^{k}P_{k-m}{D}^{(\beta)}_{\tau} (v^m)^2
& = &\sum_{m=0}^{k}{\varrho}_{k-m}\sum_{j=0}^{m}{\varpi}_{m-j}\left[(v^j)^2-(v^0)^2\right] \nonumber \\
& = &\sum_{j=0}^{k}\left[(v^j)^2-(v^0)^2\right]\sum_{m=j}^{k}{\varrho}_{k-m}{\varpi}_{m-j} \nonumber \\
& = &\sum_{j=0}^{k}{\vartheta}_{k-j}\left[(v^j)^2-(v^0)^2\right]=(v^k)^2-(v^0)^2,
\end{eqnarray}
where we exchanged the order of summation and rearranged the coefficient $\sum_{m=j}^{k}{\varrho}_{k-m}{\varpi}_{m-j}$ to ${\vartheta}_{k-j}$.

By Lemma \ref{lem2.3} and the technique for the proof of Theorem 3.1 in \cite{liao2018discrete}, we
derive the desired result, which completes the proof.
\end{proof}

We also have an alternative version of the above theorem.
\begin{cor}\label{cor:2.1} Theorem \ref{thm:2.1} remains valid if the condition \eqref{E:2.1.6} is replaced by
\begin{eqnarray}
\label{E:2.1.7}
{D}^{(\beta)}_{\tau} v^k \leq \sum_{l=1}^k \lambda_{k-l} v^l+g^{k}, \qquad {\rm for} \quad 1 \leq k \leq n_T.
\end{eqnarray}
\end{cor}
\begin{proof} Similar to the proof of the theorem 3.4 in \cite{liao2018discrete}, and by the proof of Theorem \ref{thm:2.1}, we can complete our proof.
\end{proof}

\begin{rem}\label{rem:2.1}
If $\lambda_0, \lambda_1,\cdots, \lambda_{k-1}$ are non-positive, a simple deduction will show that both Theorem \ref{thm:2.1} and
Corollary \ref{cor:2.1} hold for any $\tau>0$.
\end{rem}

\subsection{Stability and convergence}
We  present the stability and convergence result for the scheme \eqref{E:2.6}.
\begin{thm}\label{thm:2.2} Suppose that $u_N^k~(k=1,2,\cdots,n_T)$ are solutions of \eqref{E:2.6}, $f \in C(0,T;C(\bar I))$. Then for any $\tau>0$, it holds that
\begin{eqnarray}\label{E:2.2.1}
\|u_N^k\| \leq C\left(\|u^0\|+\max_{1\leq m \leq k} \sum_{j=0}^m P_{m-j}\|L_p^{(\beta)}f^{j}\|\right), \qquad   1 \leq k \leq n_T,
\end{eqnarray}
where $C$ is a positive constant  independent of $\tau$ and $N$.
\end{thm}
\begin{proof}
\begin{itemize}
  \item [{(i)}]  For $p=1$, letting $v=2u_N^k$ in \eqref{E:2.6}, one has
\begin{eqnarray}
\label{E:2.2.2}
\left({D}^{(\beta)}_{\tau} u_N^k,2u_N^k\right)+ 2\mu \left(\partial_x u_N^{k},\partial_x u_N^k\right)=\left(I_N f^{k},2u_N^k\right), \qquad   1 \leq k \leq n_T.
\end{eqnarray}
By \eqref{E:2.2}, Lemma \ref{lem2.1}, and the Young's inequality, we have
\begin{eqnarray}
\label{E:2.2.3}
\left({D}^{(\beta)}_{\tau} u_N^k,2u_N^k\right)
&  = &\frac{2}{{\tau}^{\beta}}\left[\sum_{j=0}^k {\varpi}_{k-j}\left(u_N^j,u_N^k\right)-\sum_{j=0}^k {\varpi}_{k-j}\left(u^0,u_N^k\right)\right] \nonumber \\
&  = &\frac{1}{{\tau}^{\beta}}\left[2{\varpi}_{0}\left(u_N^k,u_N^k\right)+2\sum_{j=1}^k {\varpi}_{k-j}\left(u_N^j,u_N^k\right)-2b_k\left(u^0,u_N^k\right)\right] \nonumber \\
&\geq&\frac{1}{{\tau}^{\beta}}\left[2{\varpi}_{0}\|u_N^k\|^2+\sum_{j=1}^k {\varpi}_{k-j}\|u_N^j\|^2+\sum_{j=1}^k {\varpi}_{k-j}\|u_N^k\|^2-b_k\|u_N^k\|^2-b_k\|u^0\|^2\right] \nonumber \\
&  = &\frac{1}{{\tau}^{\beta}}\left[\sum_{j=0}^k {\varpi}_{k-j}\|u_N^j\|^2-b_n\|u^0\|^2\right]={D}^{(\beta)}_{\tau} \left(\|u_N^k\|^2\right).
\end{eqnarray}

Using \eqref{E:2.2.3}, the positive-definiteness of $\left(\partial_x u_N^{k},\partial_x u_N^k\right)$,
and $\|I_N f^k\| \leq C\|f^k\|$, one has
\begin{eqnarray}
\label{E:2.2.4}
{D}^{(\beta)}_{\tau} \left(\|u_N^k\|^2\right)\leq C\|u_N^k\|~\|f^k\|, \qquad {\rm for} \quad 1 \leq k \leq n_T.
\end{eqnarray}

Finally, applying the discrete Gr{\"o}nwall inequality (see Theorem \ref{thm:2.1}) and introducing the following notations
\begin{eqnarray*}
v^k:=\|u_N^k\|,~~v^0:=u^0,~~g^{k-\theta}:=C\|f^{k}\|~({\rm with~\theta=0}),~~\lambda_j:=0~ {\rm for}~~0 \leq j \leq n_T-1,
\end{eqnarray*}
we immediately get the stability result \eqref{E:2.2.1} for $p=1$.

  \item [{(ii)}]   For $p=2$, letting $v=2u_N^{k-\beta/2}$ in \eqref{E:2.6},  one has
\begin{eqnarray}
\label{E:2.2.5}
\left({D}^{(\beta)}_{\tau} u_N^k,2u_N^{k-\beta/2}\right)+ 2\mu \left(\partial_x u_N^{k-\beta/2},\partial_x u_N^{k-\beta/2}\right)=\left(I_N f^{k-\beta/2},2u_N^{k-\beta/2}\right).
\end{eqnarray}

Rearranging the coefficients in   \eqref{E:2.2}, we have
\begin{eqnarray}
\label{E:2.2.6}
\left({D}^{(\beta)}_{\tau} u_N^k,2u_N^{k-\beta/2}\right)= \frac{1}{{\tau}^{\beta}}\sum_{j=1}^k {b}_{k-j}\left(u_N^j-u_N^{j-1},2u_N^{k-\beta/2}\right).
\end{eqnarray}
Similar to the proof of Lemma 4.1 in \cite{liao2018discrete} and by Lemma \ref{lem2.1}, we   get
\begin{eqnarray}
\label{E:2.2.7}
\left({D}^{(\beta)}_{\tau} u_N^k,2u_N^{k-\beta/2}\right) \geq {D}^{(\beta)}_{\tau} \left(\|u_N^k\|^2\right).
\end{eqnarray}
The remaining of the proof is similar to that shown in (i), which is omitted here.
The proof is completed.
\end{itemize}
\end{proof}

To obtain the convergence results, we introduce the following two lemmas.
\begin{lem}[{\cite[Theorem 6.2, p. 262]{bernardi1997spectral}}]\label{lem2.4}
Let $s$ and $r$ be arbitrary real numbers satisfying $0 \leq s \leq r$. There exist a projector $\Pi_N^{1,0}$ and a positive constant $C$ depending only on $r$ such that, for any function $u \in H_0^{s}(I) \cap H^r(I)$, the following estimate holds:
$$
\|u-\Pi_N^{1,0}u\|_{H^{s}(I)} \leq C N^{s-r} \|u\|_{H^r(I)},
$$
where the orthogonal projection operator $\Pi_N^{1,0}: H_0^{1}(I)\rightarrow V_N^0$ is defined as
$$
\left(\partial_x \left(\Pi_N^{1,0}u-u\right),\partial_x v\right)=0, \qquad \forall v \in V_N^0.
$$
\end{lem}
\begin{lem}[{\cite[Theorem 13.4, p. 303]{bernardi1997spectral}}]\label{lem2.5}
Let $r$ be arbitrary real numbers satisfying $r > 1/2$ and $I_N$ be the usual LGL interpolation operator. There exist a positive constant $C$ depending only on $r$ such that, for any function $u \in H^r(I)$, the following estimate holds:
$$
\|u-I_Nu\| \leq C N^{-r} \|u\|_{H^r(I)}.
$$
\end{lem}

Next, we consider the convergence analysis for the scheme \eqref{E:2.6}.
We also assume that the solution $u(t)$ satisfies (see, e.g., \cite{Diethelm-B10,zeng2017second})
\begin{equation}\label{solu:u}
u(t)=u_0+c_0t^{\sigma}+\sum_{j=1}^{\infty}c_kt^{\sigma_j},
{\quad}\beta<\sigma<\sigma_j<\sigma_{j+1}.
\end{equation}
The singularity index $\sigma$ determines the accuracy
of the numerical solution if $t^{\sigma}$ is not treated
properly. We do not investigate how to deal with the
singularity of the solution in this work, the interesting readers can
refer to \cite{Lubich1986Discretized,stynes2017error,zeng2017second}.
For the time discretization in \eqref{E:2.6},
the global convergence rate is $\min\{\sigma-\beta,p\}$,
that is $q=\min\{\sigma-\beta,p\}$.

Denote $u_*^k=\Pi_N^{1,0}u^k, e^k=u_*^k-u_N^k$ and $\eta^k=u^k-u_*^k$. Noticing that $\left(\partial_x \eta^k,\partial_x v\right)=0$ from Lemma \ref{lem2.4}, we get the error equation for \eqref{E:2.6}
as follows
\begin{eqnarray}
\label{E:2.2.8}
\left({D}^{(\beta)}_{\tau} e^k,v\right)+\mu \left(\partial_x e^k,\partial_x v\right)=\left(G^k,v\right),
\end{eqnarray}
where $G^k=\sum_{i=1}^3 G_i^k$ and
\begin{eqnarray}
\label{E:2.2.9}
G_1^k=f^k-F_p^k,\quad G_2^k=R^k
=O(\tau^{\sigma-\beta}k^{\sigma-\beta-p}),\quad G_3^k=-{D}^{(\beta)}_{\tau} \eta^k.
\end{eqnarray}

By Theorem \ref{thm:2.2}, Lemmas \ref{lem2.4} and   \ref{lem2.5}, we obtain the convergence result for the scheme \eqref{E:2.6}.
\begin{thm}\label{thm:2.3} Suppose that $r \geq 1$, $u$ and $u_N^k~(k=1,2,\cdots,n_T)$ are solutions of \eqref{E:1.1} and \eqref{E:2.6}, respectively. If $m \geq r+1, u\in C(0,T; H^m(I)\cap H^1_0(I)), f \in C(0,T;H^m(I))$ and $u_0 \in H^m(I)$. Then for any $\tau > 0$, it holds that
\begin{eqnarray}
\label{2.2.10}
\|u^k-u_N^k\| \leq C(\tau^q+N^{-r}),
\end{eqnarray}
where $C$ is a positive constant, independent of $\tau, N$, $q=\min\{\sigma-\beta,p\}$.
\end{thm}
\begin{proof}
We consider $p=1$, the case for $p=2$ is similarly proved.
By Theorem \ref{thm:2.2} and the fifth line of Lemma \ref{lem2.2}, we need only to evaluate
$$\|e^0\|+2C\max_{1\leq k \leq n_T} \left\{\|G_1^k\|+\|G_2^k\|+\|G_3^k\|\right\}$$
to get an error bound. By \eqref{E:2.2.9}, Lemmas \ref{lem2.2},  \ref{lem2.4}, and  \ref{lem2.5}, we get the error bounds as follows
\begin{eqnarray*}
\|G_1^k\| &\leq& CN^{-r},\qquad \|G_2^k\| \leq C\tau^q,\qquad
\|G_3^k\|=\frac{1}{{\tau}^{\beta}}\left\|\sum_{j=0}^k {\varpi}_{k-j}\left(\eta^j-\eta^0\right)\right\|
\leq CN^{-r},\\
\|e^0\| &=& \|u_*^0-I_N u_0\| \leq \|u_0-I_Nu_0\|+\|u_0-u_*^0\| \leq CN^{-r}.
\end{eqnarray*}
The above bounds yield
$$\|e^k\| \leq C(\tau^q+N^{-r}).$$
By using Lemma \ref{lem2.4} again, one has
$$\|u^k-u_N^k\| = \|u^k-u_*^k+u_*^k-u_N^k\| \leq \|\eta^k\|+\|e^k\|\leq C(\tau^q+N^{-r}).$$
The proof is completed.
\end{proof}

\section{Numerical analysis for the nonlinear equation}\setcounter{equation}{0}
In this section, we develop the semi-implicit time-stepping Legendre Galerkin spectral method for the nonlinear problem \eqref{E:1.1}, i.e., $f(x,t,u)=f(u)$. We then combine the discrete fractional Gr{\"o}nwall inequality and the temporal-spatial error splitting argument  (see, e.g., \cite{li2012mathematical,li2017unconditionally})
to prove the stability and convergence of the numerical scheme.
We assume that the solution of problem \eqref{E:1.1} satisfies the following condition
\begin{eqnarray}
\label{E:3.1}
\|u_0\|_{H^{r+1}}+\|u\|_{L^{\infty}(0,T;H^{r+1})} \leq K,
\end{eqnarray}
where $K$ is a positive constant independent of $N$ and $\tau$.

Throughout this section,  $C$ denotes a generic positive constant, which may vary at different occurrences, but  is  independent of $N$ and time step size $\tau$. Denote $C_i~(i=1,2,\cdots)$ as positive constants independent of $N$ and   $\tau$.

The following inverse inequalities
(see, e.g., \cite{Adams03,li2003legendre,brenner2007mathematical})
will be  used in the  numerical analysis
\begin{eqnarray}
&&\|v\|_{L^{\infty}} \leq \frac{N+1}{\sqrt{2}} \|v\|,  \qquad \forall~ v \in V_N^0,\label{E:3.2}\\
&&\|v\|_{L^{\infty}} \leq C_I \|v\|_{H^2},  \qquad \forall~ v \in H^1_0(I)\cap H^2(I),\label{E:3.3}
\end{eqnarray}
where $C_I$ is a positive constant  depending only on the interval $I$.

We extend the method \eqref{E:2.6} with $p=1$ to the nonlinear equation \eqref{E:1.1},
while the nonlinear term  $f(u_N^k)$ is approximated  by $f(u_N^{k-1})$.
We obtain the following semi-implicit  Galerkin spectral method:
find $u_N^k\in V_N^0$ such that
\begin{equation}\label{E:3.7}
\left\{\begin{array}{ll}
\left({D}^{(\beta)}_{\tau} u_N^k,v\right)+ \mu \left(\partial_x u_N^{k},\partial_x v\right)= \left(f(u_N^{k-1}),v\right),\quad \forall v \in V_N^0, \quad k=1,2,\cdots,n_T,\\ [0.1cm]
u^0=I_N u_0,
\end{array}\right.
\end{equation}
where $I_N$ is the LGL interpolation operator.



We also assume that the nonlinear term $f(u)$ satisfies the local Lipschitz condition
$$
\|f(u^{k})-f(u^{k-1})\| \leq \|f'(\xi)\|~\|u^{k}-u^{k-1}\|\leq C \|u^{k}-u^{k-1}\|,
{\quad} |\xi|\leq K_1,
$$
where $K_1$ is a positive constant that is suitably large.

%

\subsection{An error estimate of the time discrete system}
In order to obtain the unconditionally stability of \eqref{E:3.7}, we now introduce a time-discrete system
\begin{equation}
\label{E:3.1.1}
\left\{\begin{array}{lll}
{D}^{(\beta)}_{\tau} U^k=\mu \partial^2_x U^{k}+f(U^{k-1}),\qquad k=1,2,\cdots,n_T,\\
U^k(x)=0,\qquad {\rm for} \quad x \in \partial I, \quad k=1,2,\cdots,n_T,\\
U^0(x)=u_0(x),\qquad {\rm for} \quad x \in I.
\end{array}\right.
\end{equation}

Let $R_1^k$ be the time discretization  error of \eqref{E:3.7}. Then we can obtain
\begin{eqnarray}\label{E:3.4}
{D}^{(\beta)}_{\tau} u^k=\mu \partial^2_x u^{k}+f(u^{k-1})+R_1^k,\qquad k=1,2,\cdots,n_T,
\end{eqnarray}
where
\begin{eqnarray}
\label{E:3.5}
R_1^k = {D}^{(\beta)}_{\tau} u^k-{_{0}^{C}{\mathcal {D}}}^{\beta}_t u^k+f(u^{k})-f(u^{k-1})
=O(\tau^{\tilde{q}}),
\end{eqnarray}
in which $\tilde{q}=\min\{\sigma-\beta,1\}$ when the first-order extrapolation is applied.

Letting ${\varepsilon}^k=u^k-U^k$ and subtracting \eqref{E:3.1.1} from \eqref{E:3.4} gives
\begin{eqnarray}
\label{E:3.1.4}
{D}^{(\beta)}_{\tau} {\varepsilon}^k=\mu \partial^2_x {\varepsilon}^{k}+f(u^{k-1})-f(U^{k-1})+R_1^k,\qquad k=1,2,\cdots,n_T.
\end{eqnarray}

Define
$$K_1=\max_{1\leq k \leq n_T} \|u^k\|_{L^{\infty}}+\max_{1\leq k \leq n_T} \|u^k\|_{H^2}+\max_{1\leq k \leq n_T} \|{D}^{(\beta)}_{\tau}u^k\|_{H^2}+1.$$
We now present an error bound of ${\varepsilon}^k=u^k-U^k$   as follows.

\begin{thm}\label{thm:3.1} Suppose that $r \geq 1$, $u$ and $U^k~(k=1,2,\cdots,n_T)$
are solutions of \eqref{E:1.1} and \eqref{E:3.1.1}, respectively.
If $u\in C(0,T; H^2(I)\cap H^1_0(I))$, $|f'(z)|$ is bounded for $|z|\leq K_1$, $f\in H^1(I)$, and $u_0 \in H^m(I)$. Then there exists a suitable
constant $\tau_0^*>0$ such that when $\tau \leq \tau_0^*$, it holds
\begin{eqnarray}
\|{\varepsilon}^k\|_{H^2} &\leq& C_*\tau^{\tilde{q}},\label{E:3.1.5}\\
\|U^k\|_{L^{\infty}}+\|{D}^{(\beta)}_{\tau}U^k\|_{H^2} &\leq& 2K_1,\label{E:3.1.6}
\end{eqnarray}
where $C_*$ is a constant independent of $N$ and $\tau$,
and $\tilde{q}=\min\{\sigma-\beta,1\}$.
\end{thm}
\begin{proof}
We  use the mathematical induction method to prove \eqref{E:3.1.5}. Obviously, \eqref{E:3.1.5} holds for $k=0$. Assume \eqref{E:3.1.5} holds for $k \leq n-1$. Then, by  \eqref{E:3.2}
and \eqref{E:3.1.5}, one has
\begin{eqnarray}
\label{E:3.1.7}
\|U^k\|_{\infty}
\leq \|u^k\|_{L^{\infty}}+C_I\|{\varepsilon}^k\|_{H^2}
\leq \|u^k\|_{L^{\infty}}+C_IC_*\tau^{\tilde{q}}
\leq K_1\qquad k \leq n-1,
\end{eqnarray}
provided $\tau \leq (C_IC_*)^{-1/\tilde{q}}$. Moreover, for $\tau\leq (C_*)^{-1/\tilde{q}}$, we have
\begin{eqnarray}
\label{E:3.1.8}
\|U^k\|_{H^2} \leq \|u^k\|_{H^2}+\|{\varepsilon}^k\|_{H^2} \leq \|u^k\|_{H^2}+C_*\tau^{\tilde{q}}
\leq K_1.\end{eqnarray}

The following estimates are easily obtained:
\begin{eqnarray}
\|f(u^{k-1})-f(U^{k-1})\| &\leq&  C\|{\varepsilon}^{k-1}\|, \label{E:3.1.9}\\
 \left\|\partial_x \left(f(u^{k-1})-f(U^{k-1})\right)\right\|
&\leq&C\|{\varepsilon}^{k-1}\|_{H^1}.\label{E:3.1.10}
\end{eqnarray}

Next, we prove that \eqref{E:3.1.5} holds for $k\leq n$ in \eqref{E:3.1.4}.
Multiplying both sides of \eqref{E:3.1.4} by $2{\varepsilon}^{n}$ and
integrating the result over $I$ yields
\begin{eqnarray}
\label{E:3.1.12}
{D}^{(\beta)}_{\tau} \left(\|{\varepsilon}^n\|^2\right)
& \leq &2\left(f(u^{k-1})-f(U^{k-1}),{\varepsilon}^{n}\right)
+2\left(R_1^n,{\varepsilon}^{n}\right) \nonumber \\
& \leq & C\|{\varepsilon}^{n}\|^2+C\|{\varepsilon}^{n-1}\|^2
+2\|{\varepsilon}^{n}\|~\|R_1^n\|,
\end{eqnarray}
where \eqref{E:3.1.9} is used.

Applying \eqref{E:3.5} and  Theorem \ref{thm:2.1} with
\begin{eqnarray*}
v^n:=\|{\varepsilon}^{n}\|,\quad v^0:=0,\quad g^{n-\theta}:=2\|R_1^n\|~({\rm with~\theta=0}), \\
\lambda_0=\lambda_1:=C,\quad \lambda_j:=0\quad {\rm for}\quad 2 \leq j \leq n_T-1,
\end{eqnarray*}
one has
\begin{eqnarray}
\label{E:3.1.13}
\|{\varepsilon}^{n}\| \leq C_1 \tau^{\tilde{q}},{\quad} \text{if} \quad
\tau\leq \frac{1}{\sqrt[\beta]{C(1+\beta)}}.
\end{eqnarray}

To derive an estimate of $\|\partial_x {\varepsilon}^{n}\|$, we multiply \eqref{E:3.1.4} by $2{D}^{(\beta)}_{\tau}{\varepsilon}^{n}$  and integrate the result over $I$ to obtain
\begin{eqnarray}
\label{E:3.1.14}
& &2\|{D}^{(\beta)}_{\tau} {\varepsilon}^n\|^2+\mu {D}^{(\beta)}_{\tau} \left(\|\partial_x{\varepsilon}^n\|^2\right)=2\left(f(u^{n-1})-f(U^{n-1}),{D}^{(\beta)}_{\tau}{\varepsilon}^{n}\right)
+2\left(R_1^n,{D}^{(\beta)}_{\tau}{\varepsilon}^{n}\right).
\end{eqnarray}


By Young inequality, \eqref{E:3.5}, \eqref{E:3.1.9}, and \eqref{E:3.1.13}, we can derive
\begin{eqnarray*}
\left|\left(f(u^{n-1})-f(U^{n-1}),{D}^{(\beta)}_{\tau}{\varepsilon}^{n}\right)\right|
&\leq& \frac{3}{2}\left\|f(u^{n-1})-f(U^{n-1})\right\|^2+\frac{2}{3}\|{D}^{(\beta)}_{\tau} {\varepsilon}^n\|^2
\leq \frac{2}{3}\|{D}^{(\beta)}_{\tau} {\varepsilon}^n\|^2 +C{\tau}^{2\tilde{q}},\\
\left|\left(R_1^n,{D}^{(\beta)}_{\tau}{\varepsilon}^{n}\right)\right|
&\leq& \frac{3}{2}\left\|R_1^n\right\|^2+\frac{2}{3}\|{D}^{(\beta)}_{\tau} {\varepsilon}^n\|^2
\leq \frac{2}{3}\|{D}^{(\beta)}_{\tau} {\varepsilon}^n\|^2 +C{\tau}^{2\tilde{q}}.
\end{eqnarray*}
Substituting the above estimates into \eqref{E:3.1.14}, we have

\begin{eqnarray}
\label{E:3.1.15}
{D}^{(\beta)}_{\tau} \left(\|\partial_x{\varepsilon}^n\|^2\right) \leq C {\tau}^{2\tilde{q}}.
\end{eqnarray}

Applying   Corollary \ref{cor:2.1}  with
\begin{eqnarray*}
v^n:=\|\partial_x {\varepsilon}^{n}\|^2,\quad v^0:=0,\quad g^{n}:=C {\tau}^{2\tilde{q}},
\quad \lambda_j:=0\quad {\rm for}\quad 0 \leq j \leq n_T-1,
\end{eqnarray*}
one has
\begin{eqnarray}
\label{E:3.1.16}
\|\partial_x{\varepsilon}^{n}\| \leq C_2 \tau^{\tilde{q}}.
\end{eqnarray}

We can similarly derive an estimate of $\|\partial^2_x{\varepsilon}^{n}\|$ by multiplying \eqref{E:3.1.4} by $-2{D}^{(\beta)}_{\tau}\left(\partial_x^2{\varepsilon}^{n}\right)$ and integrating the result over $I$. Similar to   \eqref{E:3.1.16}, by Young inequality, \eqref{E:3.5}, \eqref{E:3.1.9}, \eqref{E:3.1.10}, \eqref{E:3.1.13} and \eqref{E:3.1.16}, we can get
\begin{eqnarray}
\label{E:3.1.17}
\|\partial^2_x{\varepsilon}^{n}\| \leq C_3 \tau^{\tilde{q}}.
\end{eqnarray}

Combing \eqref{E:3.1.13}, \eqref{E:3.1.16} and \eqref{E:3.1.17}, we obtain
\begin{eqnarray}
\label{E:E:3.1.18}
\|{\varepsilon}^{n}\|_{H^2} \leq C_*\tau^{\tilde{q}},
\end{eqnarray}
where $C_*=\sqrt{C_1^2+C_2^2+C_3^2}$ is a constant independent of $N$ and $\tau$.

Moreover, we can derive that
\begin{eqnarray*}
\|U^{n}\|_{L^{\infty}} &\leq& \|u^{n}\|_{L^{\infty}}+C_I\|\varepsilon^{n}\|_{H^2} \leq \|u^{n}\|_{L^{\infty}}+C_*C_I\tau \leq K_1,\\
\|{D}^{(\beta)}_{\tau} U^{n}\|_{H^2} &\leq& \|{D}^{(\beta)}_{\tau} u^{n}\|_{H^2}+\|{D}^{(\beta)}_{\tau} \varepsilon^{n}\|_{H^2} \leq \|{D}^{(\beta)}_{\tau} u^{n}\|_{H^2}+C_*{\tau}^{\tilde{q}-\beta}\leq  K_1,
\end{eqnarray*}
where $\tau \leq \tau_0= \min \left\{(C_*C_I)^{-1},C_*^{\frac{1}{1-\beta}}\right\}$, and \eqref{E:2.2} is used. Thus, the proof is completed.
\end{proof}

\subsection{An  error estimate of the space discrete system}
The weak form of time-discrete system \eqref{E:3.1.1} satisfies
\begin{eqnarray}
\label{E:3.2.1}
\left({D}^{(\beta)}_{\tau} U^k,v\right)=\mu \left(\partial^2_x U^{k},v\right)+\left(f(U^{k-1}),v\right),\qquad v \in H^2(I).
\end{eqnarray}
Let
$$U^k_*=\Pi_N^{1,0}U^k, \quad {\bar e}^k=U^k_*-u_N^k, \quad k=1,2,\cdots,n_T.$$
Subtracting \eqref{E:3.7} from \eqref{E:3.2.1}, we have
\begin{eqnarray}
\label{E:3.2.2}
\left({D}^{(\beta)}_{\tau} {\bar e}^k,v\right)+\mu \left(\partial_x {\bar e}^{k},\partial_x v\right)=\left(f(U^{k-1})-f(u_N^{k-1}),v\right)+\left(R_2^k,v\right),
\end{eqnarray}
where
$$
R_2^k={D}^{(\beta)}_{\tau} (U_*^k-U^k).
$$

It is easy to obtain $\left\|\Pi_N^{1,0}v\right\|_{L^{\infty}}\leq C\|v\|_{H^2}$ for any $v \in H^2(I)$. By Theorem \ref{thm:3.1}, one has
$$\|U_*^k\|_{L^{\infty}}\leq C\|U^k\|_{H^2} \leq C,\qquad k=1,2,\cdots,n_T.$$
Then, we   define
$$K_2=\max_{1\leq k \leq n_T}\|U_*^k\|_{L^{\infty}}+1.$$

Now, we are ready to give an error estimate of $\|U^k-u_N^k\|$.
\begin{thm}\label{thm:3.2} Suppose that $r \geq 1$, $u_N^k$ and $U^k~(k=1,2,\cdots,n_T)$ are solutions of \eqref{E:3.7} and \eqref{E:3.1.1}, respectively.
Assume that $U^k\in H^2(I)\cap H^1_0(I)$, $|f'(z)|$ is bounded for $|z|\leq K_1$, $f\in H^1(I)$, and $U^0 \in H^2(I)$.
Then there exists a  positive constant $N_0^*$ such that
when $N \geq N_0^*$, it holds
\begin{eqnarray}
\|U^k-u_N^k\| &\leq& N^{-\frac{3}{2}},\label{E:3.2.3}\\
\|u_N^k\|_{L^{\infty}} &\leq& K_2.\label{E:3.2.4}
\end{eqnarray}
\end{thm}
\begin{proof} We prove \eqref{E:3.2.3} by using the mathematic induction method.
From Lemma \ref{lem2.5}, one has
$$\|U^0-u_N^0\|=\|u_0-I_Nu_0\| \leq C_4N^{-2} \leq N^{-\frac{3}{2}}$$
when $N\geq (C_4)^2$. Assume that \eqref{E:3.2.3} holds for $k \leq n-1$. By the inverse inequality \eqref{E:3.2}, we have
\begin{eqnarray*}
\|u_N^k\|_{L^{\infty}}\leq \|U_*^k\|_{L^{\infty}}+\|U_*^k-u_N^k\|_{L^{\infty}} \leq \|U_*^k\|_{L^{\infty}}+\frac{N+1}{\sqrt{2}}\|{\bar e}^k\| \leq \|U_*^k\|_{L^{\infty}}+C_5N^{-\frac{1}{2}} \leq K_2,
\end{eqnarray*}
when $N\geq (C_5)^2$.

By \eqref{E:3.1.6} and the assumption,
$\|U^{k}\|_{L^{\infty}}$ and $\|u_N^{k}\|_{L^{\infty}}$
are bounded for $k\leq n-1$. Therefore,   $f'(\xi)$ is bounded when
 $|\xi|\leq \max\{2K_1,K_2\}$. Combining Lemma \ref{lem2.4} and the boundedess of $f'(\xi)$
 yields
\begin{eqnarray}
\label{E:3.2.5}
\left\|f(U^{k-1})-f(u_N^{k-1})\right\|= \|f'(\xi)(u_N^{k-1}-U^{k-1})\|
\leq C\left\|U^{k-1}-u_N^{k-1}\right\|
\leq C\|{\bar e}^{k-1}\|+CN^{-2}.
\end{eqnarray}
By Lemma \ref{lem2.4} and Theorem \ref{thm:3.1}, we obtain
\begin{eqnarray}
\label{E:3.2.6}
\|R_2^n\|^2 \leq CN^{-4}\left\|{D}^{(\beta)}_{\tau}U^k \right\|_{H^2} \leq CN^{-4}.
\end{eqnarray}

Letting $k\leq n$ and $v=2{\bar e}^{k}$ in \eqref{E:3.2.2}, we have
\begin{eqnarray*}
\left({D}^{(\beta)}_{\tau} {\bar e}^k,2{\bar e}^k\right)+2\mu
\left(\partial_x {\bar e}^{k},\partial_x {\bar e}^k\right)
=2\left(f(U^{k-1})-f(u_N^{k-1}),
{\bar e}^k\right)+2\left(R_2^k,{\bar e}^k\right).
\end{eqnarray*}
Combing \eqref{E:3.2.5}, \eqref{E:3.2.6}, and the above inequality yields
\begin{eqnarray}
\label{E:3.2.7}
{D}^{(\beta)}_{\tau} \left(\|{\bar e}^k\|^2\right)
& \leq & 2\left(f(U^{k-1})-f(u_N^{k-1}),
{\bar e}^k\right)+2\left(R_2^k,{\bar e}^k\right) \nonumber \\
& \leq & \left(\left\|f(U^{k-1})-f(u_N^{k-1})\right\|^2
+\|{\bar e}^k\|^2\right)+\left(\|R_2^k\|^2+\|{\bar e}^k\|^2\right) \nonumber \\
& \leq &2\|{\bar e}^k\|^2+C\|{\bar e}^{k-1}\|^2+CN^{-4},{\quad} k\leq n.
\end{eqnarray}
Applying Theorem \ref{thm:2.1} yields $\|{\bar e}^k\| \leq CN^{-2}(k\leq n)$, which leads to
\begin{eqnarray}
\label{E:3.2.9}
\|U^n-u_N^n\|\leq \|U^n-U_*^n\|+\|{\bar e}^n\| \leq C_6N^{-2} \leq N^{-\frac{3}{2}},
\end{eqnarray}
when $N \geq (C_6)^2$. That is to say, \eqref{E:3.2.3} holds for $k=n$. Furthermore, we have
\begin{eqnarray}
\label{E:3.2.10}
\|u_N^n\|_{L^{\infty}}
 \leq  \|U_*^n\|_{L^{\infty}}+\|{\bar e}^n\|_{L^{\infty}}
 \leq  \|U_*^k\|_{L^{\infty}}+\frac{N+1}{\sqrt{2}}\|{\bar e}^n\|
 \leq \|U_*^k\|_{L^{\infty}}+C_5N^{-\frac{1}{2}}\leq K_2,
\end{eqnarray}
when $N \geq (C_5)^2$. Letting $N_0^*=\left\lceil\max\left\{(C_4)^2,(C_5)^2,(C_6)^2\right\}\right\rceil$
completes the proof.
\end{proof}

\subsection{Error estimate of the fully discrete system}
By the boundedness of $u_N^k$ and Theorem \ref{thm:2.2}, we immediately obtain the following result.
\begin{thm}\label{thm:3.3} Suppose that $r \geq 1$, $u$ and $u_N^k~(k=1,2,\cdots,n_T)$ are solutions of \eqref{E:1.1} and \eqref{E:3.7}, respectively. If $m \geq r+1, u\in C(0,T; H^m(I)\cap H^1_0(I))$ and satisfies \eqref{solu:u}, $|f'(z)|$ is bounded for $|z|\leq K_1$, $f\in H^1(I)$, and $u_0 \in H^m(I)$.  Then, there exist two positive constants $\tau^*_0$ and $N^*_0$ such that
when $\tau \leq \tau^*_0$ and $N \geq N^*_0$, it holds
\begin{eqnarray}
\label{E:3.3.1}
\|u^k-u_N^k\| \leq C\left(\tau^{\min\{\sigma-\beta,1\}}+N^{-r}\right).
\end{eqnarray}
\end{thm}

\begin{rem}\label{rem:3.1}
We can also extend the method \eqref{E:2.6} with $p=2$ to the nonlinear equation \eqref{E:1.1},
and the nonlinear term $f(u_N^k)$ is approximated by a second-order extrapolation
$f(2u_N^{k-1}-u_N^{k-2})$. The numerical method is given by: find $u_N^k\in V_N^0$ for $k\geq2$ such that
\begin{equation}\label{E:3.3.5}
\left\{\begin{array}{ll}
\left({D}^{(\beta)}_{\tau} u_N^k,v\right)+ \mu \left(\partial_x u_N^{k-\beta/2},\partial_x v\right)=\left((1-\frac{\beta}{2})f(2u_N^{k-1}-u_N^{k-2})+\frac{\beta}{2}f(u_N^{k-1}),v\right),\quad \forall v \in V_N^0,\\ [0.1cm]
u^0=I_N u_0,
\end{array}\right.
\end{equation}
where $u_N^1$ can be derived by the fully implicit method or \eqref{E:3.7} with a smaller
step size.
\end{rem}
The stability and convergence analysis of the method \eqref{E:3.3.5}
is similar to that of  \eqref{E:2.6}.
\begin{thm}\label{thm:3.4} Suppose that $r \geq 1$, $u$ and $u_N^k~(k=1,2,\cdots,n_T)$ are
solutions of \eqref{E:1.1} and \eqref{E:3.3.5}, respectively. Assume that
$m \geq r+1, u\in C(0,T; H^m(I)\cap H^1_0(I))$ and satisfies \eqref{solu:u},
$|f'(z)|$ and $|f''(z)|$ are bounded for $|z|\leq K_1$, $f\in H^1(I)$, and $u_0 \in H^m(I)$.
Then, there exist two  positive constants $\tau_1^*, N^*_1$ such that when $\tau \leq \tau^*_1$ and $N \geq N^*_1$,  it holds
\begin{eqnarray}
\|u^k-u_N^k\| \leq C\left(\tau^{\min\{\sigma-\beta,2\}}+N^{-r}\right).\label{E:3.3.6}
\end{eqnarray}
\end{thm}

\section{Numerical results}
\setcounter{equation}{0}

In this section, a numerical example is presented to illustrate the proposed method.

Consider the model problem \eqref{E:1.1} with $\mu=1$ and $f(x,t,u)=u+u^2$. The initial condition is chosen as
$$u_0(x)=\sin(2\pi x).$$

Since the exact solution of the problem is unknown, the reference solutions are derived by setting $N=2^9, \tau=1/2^{12}$. In Table \ref{s4:tb1}, we list the $L^2$-errors and convergence rates of the method \eqref{E:3.7} in temporal direction with $N=2^9$ and different $\beta$. From Table \ref{s4:tb1}, we can observe the first-order accuracy in time at $t=1$ for $\beta=0.2$ and $\beta=0.9$. In Table \ref{s4:tb2}, we show the $L^2$-errors and convergence rates
of the method \eqref{E:3.3.5} in temporal direction for  $N=2^9$. From Table \ref{s4:tb2}, a second convergence order is obtained for $\beta=0.9$
due to relatively good regularity of the solution. However, we do  not observe second-order accuracy  for $\beta=0.2$ due to slightly stronger singularity of the solution, but
second-order convergence can be recovered by adding the correction terms, which is not
investigated here; see \cite{Lubich1986Discretized,zeng2017second}.

\begin{table}[!h]
\caption{The $L^2$-errors and convergence rate of the method \eqref{E:3.7}  in time.}\label{s4:tb1}
\centering
\begin{tabular}{|c|c|c|c|c|c|c|c|c|c|c|c|c|}
\hline
 $\tau$ & $\beta=0.2$& Order& $\beta=0.9$ & Order  \\
 \hline
$2^{-5}$&3.6747e-3&     &1.85544e-2&    \\
$2^{-6}$&1.7904e-3&1.03 &9.22270e-3&1.00\\
$2^{-7}$&8.7440e-4&1.03 &4.54197e-3&1.02\\
$2^{-8}$&4.2187e-4&1.05 &2.19854e-3&1.04\\
$2^{-9}$&1.9670e-4&1.10 &1.02610e-3&1.09\\
\hline
\end{tabular}
\end{table}

\begin{table}[!h]
\caption{The $L^2$-errors and convergence rate of the method \eqref{E:3.3.5} in time.}\label{s4:tb2}
\centering
\begin{tabular}{|c|c|c|c|c|c|c|c|c|c|c|c|c|}
\hline
 $\tau$ & $\beta=0.2$& Order& $\beta=0.9$ & Order \\
 \hline
$2^{-5}$&7.9765e-4&     &4.9601e-03&    \\
$2^{-6}$&3.6951e-4&1.11 &1.2854e-03&1.94\\
$2^{-7}$&1.7264e-4&1.09 &3.3088e-04&1.95\\
$2^{-8}$&7.9981e-5&1.11 &8.4300e-05&1.97\\
$2^{-9}$&3.5914e-5&1.15 &2.1181e-05&1.99\\
\hline
\end{tabular}
\end{table}

\section{Conclusion}
A discrete fractional Gr{\"o}nwall inequality for convolution quadrature with the convolution coefficients generated by the generating function is developed. We illustrate its use through the stability and convergence analysis of the Galerkin spectral method for the linear time-fractional subdiffusion equations. We then combined the discrete fractional Gr{\"o}nwall inequality and the temporal-spatial error splitting argument \cite{li2012mathematical} to prove the unconditional convergence of the Galerkin spectral method  for the nonlinear time-fractional subdiffusion equation.

We only developed a discrete fractional Gr{\"o}nwall inequality for the convolution quadrature
with the coefficients generated by the generalized Newton--Gregory formula of order up to
order two \cite{zeng2013use}. How to construct a discrete fractional Gr{\"o}nwall inequality for other convolution quadratures (see, e.g., \cite{Lubich1986Discretized}) of high-order accuracy will be considered in our future work.
It will be interesting to consider the discrete fractional Gr{\"o}nwall inequality for analyzing the numerical methods for multi-term nonlinear time-fractional differential equations \cite{zeng2017second}.


\section*{Acknowledgements}
 The authors wish to thank the referees for their constructive
comments and suggestions, which greatly improved the quality of this paper.


\end{document}